\documentclass[preprint,3p]{elsarticle}
\biboptions{sort&compress}

\usepackage{hyperref}

\usepackage[pagewise,mathlines]{lineno}

\journal{arXiv}

\usepackage{amsfonts}
\usepackage{amssymb}
\usepackage[utf8]{inputenc}
\usepackage{amsmath}
\usepackage{amsthm}
\usepackage{graphicx}%
\usepackage{enumerate}
\usepackage{xcolor}
\usepackage{color,soul}
\usepackage{tikz}
\usetikzlibrary{patterns}
\usepackage{mathtools}

\newtheorem{thm}{Theorem}[section]

\newtheorem{lem}[thm]{Lemma}
\newtheorem{prop}[thm]{Proposition}

\newtheorem*{assumption*}{\assumptionnumber}
\providecommand{\assumptionnumber}{}
\makeatletter
\newenvironment{assumption}[1]
{%
	\renewcommand{\assumptionnumber}{A.#1}%
	\begin{assumption*}%
		\protected@edef\@currentlabel{A.#1}%
	}
	{%
	\end{assumption*}
}

\theoremstyle{definition}
\newtheorem{defn}[thm]{Definition}
\theoremstyle{remark}
\newtheorem{rem}[thm]{Remark}
\newtheorem*{ex}{Example}
\numberwithin{equation}{section}

\newcommand{\ep}{\varepsilon}

\newcommand{\R}{\mathbb{R}}				      

\newcommand{\lnum}{ }
\newcommand{\dps}{\displaystyle}	
\newcommand{\intq}{\int_0^T\int_0^1}	
\newcommand{\intw}{\int_0^T\int_{\omega}}	
\newcommand{\intwl}{\int_0^T\int_{\omega'}}	
\newcommand{\into}{\int_0^1}	
\newcommand{\inta}{\int_0^T\int_0^{\alpha'}}
\newcommand{\intb}{\int_0^T\int_{\beta'}^1}
\newcommand{\dom}{Q}
\newcommand{\domw}{Q_\omega}

\newcommand{\nn}[1]{\|#1\|}

\newcommand{\wei}{e^{2s\varphi}}

\newcommand{\n}[2]{\|#1\|_{_{#2}}}

\usepackage[textsize=scriptsize,textwidth=15mm]{todonotes}
\usepackage{marginnote}
\makeatletter
\renewcommand{\@todonotes@drawMarginNoteWithLine}{%
	\begin{tikzpicture}[remember picture, overlay, baseline=-0.75ex]%
	\node [coordinate] (inText) {};%
	\end{tikzpicture}%
	\marginnote[{
		\@todonotes@drawMarginNote%
		\@todonotes@drawLineToLeftMargin%
	}]{
		\@todonotes@drawMarginNote%
		\@todonotes@drawLineToRightMargin%
	}%
}
\makeatother

\allowdisplaybreaks[1] 

\begin{document}

\begin{frontmatter}

\title{\textsc{carleman inequality for a linear degenerate parabolic problem}}

	\author[RCN]{R. Demarque\corref{mycorrespondingauthor}}
\cortext[mycorrespondingauthor]{Corresponding author}
\ead{reginaldo@id.uff.br}


\author[GMA]{J. Límaco}
\ead{jlimaco@id.uff.br}

\author[GAN]{L. Viana}
\ead{luizviana@id.uff.br}

\address[RCN]{Departamento de Ciências da Natureza,
	Universidade Federal Fluminense,
	Rio das Ostras, RJ, 28895-532, Brazil}
\address[GMA]{  Departamento de Matemática Aplicada,
	Universidade Federal Fluminense,
	Niterói, RJ, 24020-140, Brazil}
\address[GAN]{Departamento de Análise,
	Universidade Federal Fluminense,
	Niter\'{o}i, RJ, 24020-140, Brazil}

\begin{abstract}
In this work, we prove a Carleman estimate for a parabolic problem which has a dissipative degenerate term. The prove relies on choose  a suitable  weight function that change of sign inside the control domain.
\end{abstract}


\end{frontmatter}


\section{Introduction\label{intro}}

\noindent 

Let us consider the degenerate parabolic problem
\begin{equation}\lnum \label{pb-lin}
\left\{\begin{array}{ll}
u_t-\left(a\left(x\right)u_x \right)_x+c(t,x)u=h\chi_\omega, & (t,x)\in \dom; \\
\begin{cases*}
u(t,0)=0, & \\
\text{or}\\
(au_x)(t,0)=0,& 
\end{cases*},&  t\in (0,T), \\
u(0,x)=u_0(x),& x\in (0,1),
\end{array}\right.
\end{equation}
where $T>0$ is given, $\dom:=(0,T)\times (0,1)$, $\omega=(\alpha,\beta)\subset\subset (0,1)$, $u_0\in L^2(0,1)$ and $h\in L^2(\domw)$ is a control that acts on the system through $\domw:=(0,T)\times \omega$.  We also specify some properties of $a$.

\begin{assumption}{1} \label{hip1}
	Let  $a \in C([0,1])\cap C^1((0,1])$ be a nondecreasing function satisfying $a(0)=0$ and $a>0$ on $(0,1]$. Additionally, we suppose that there exist a $K\in \R$ such that  
	\begin{equation}\label{prop_a}
	xa'(x)\leq Ka(x),\ \ \forall x\in [0,1],
	\end{equation}
	where $K\in [0,1)$, for the \textbf{Weak Degeneracy Case} (WDC), and $K\in [1,2)$, for the \textbf{Strong Degeneracy} one (SDC). Only for the $(SDC)$, we also assume that
	\begin{equation}\label{teta}
	\begin{cases}
	\exists \theta \in (1,K] \text{ such that } \theta a\leq xa' \text{ near zero, if } K>1;\\
	\exists \theta \in (0, 1)  \text{ such that } \theta a\leq xa' \text{ near zero, if } K=1.\\
	\end{cases}
	\end{equation}
\end{assumption}	

Under these notations, we provide some examples and comments about Hypotheses \ref{hip1}.
\begin{ex} 
	\noindent 
	
	\begin{itemize}
		\item[(a)] Take $\gamma \in (0,1)$ and $\alpha \geq 0$. Putting $\beta = \arctan (\alpha)$, the function $a_{1} (x) = x^{\gamma} \cos (\beta x)$ fulfills \eqref{prop_a} for the (WDC). On the other hand, if $\gamma \in (1,2)$, then $a_1$ becomes an example for the (SDC); 
		
		\item[(b)] For each $\theta \in (0,1)$, the function $a_{2} (x) = x^{\theta} - x$ satisfies \eqref{prop_a} for the (WDC). However, if $\theta \in (1,2)$, then $a_{3} (x) = x^{\theta} + x$ satisfies \eqref{prop_a} for the (SDC).
	\end{itemize}
\end{ex}


It is well known that the null-controllability for \eqref{pb-lin} is a consequence of an Observability inequality which in turn is a consequence of a Carleman Estimate for the following adjoint system associated to \eqref{pb-lin}
\begin{equation}\lnum \label{adj-jlr}
\left\{\begin{array}{ll}
v_t+\left(a\left(x\right)v_x \right)_x+c(t,x)v=F, & (t,x)\in \dom, \\
\begin{cases*}
v(t,0)=0, & \\
\text{or}\\
(av_x)(t,0)=0,& 
\end{cases*},&  t\in (0,T), \\
v(T,x)=v_T(x), & x\in (0,1),
\end{array}\right.
\end{equation}
where $F\in L^2(\dom)$ and $v_T\in L^2(0,1)$. 

Alabau-Boussouira et al.  obtained a Carleman inequality to \eqref{adj-jlr} in \cite{alabau2006carleman} and proved null-controllability results to the linear and semilinear problems. However, their Carleman inequality can not be used to proved other controllability results with the same kind of degeneracy, for instance Stalkelberg-Nash null controllability  or null-controllability for nonlinear problems.

Araruna et al.  \cite{araruna2018stackelberg} proved a new Carleman estimate to \eqref{adj-jlr} when $a(x)=x^\alpha$, $\alpha \in (0,2)$ and proved a Stalkelberg-Nash null controllability result. In order to do that they  choose  a suitable  weight function that change of sign inside the control domain. Following this ideas, in \cite{jrl2016,jrl2020EECT}, the authors extended their Carleman Inequality for a general $a=a(x)$ satisfying hypotheses \eqref{A1}, but just for the weak case, and proved a null-controllability result for a degenerate problem with nonlocal nonlinearities. The aim of the present work is extend the Carleman Inequality proved in \cite{jrl2016} to the strong case. 

 In order to state our main result let us  consider $\omega'=(\alpha',\beta')\subset\subset \omega$ and $\psi \in C^2 ([0,1];\R )$ satisfying  
\begin{align}\label{functions1}
\psi(x):=\begin{cases}
\displaystyle \phantom{-}\int_0^x \frac{y}{a(y)}dy,\ x\in [0,\alpha') \vspace{.2cm} \\
\displaystyle -\int_{\beta'}^x \frac{y}{a(y)}dy,\ x\in [\beta',1].
\end{cases}
\end{align}
Setting
\begin{multline}\label{functions}
\theta(t):=\frac{1}{[t(T-t)]^4}, \ \eta(x):=e^{\lambda(|\psi|_\infty+\psi)},\ \sigma(x,t):=\theta(t)\eta(x) \mbox{ and }\\
\varphi(x,t):=\theta(t)(e^{\lambda(|\psi|_\infty+\psi)}-e^{3\lambda|\psi|_\infty}),
\end{multline}
where $(t,x) \in (0,T)\times [0,1]$ and $\lambda >0$.

\begin{thm}[Carleman Inequality]\label{car-strong}
	There exist $C>0$ and $\lambda_0,s_0>0$ such that every solution  $v$ of (\ref{adj-jlr}) satisfies, for all $s\geq s_0$ and $\lambda\geq \lambda_0$,
	\begin{equation}
	\intq e^{2s\varphi}\left((s\lambda)\sigma av_x^2+(s\lambda)^{5/3}\sigma^{5/3}v^2 \right) 
	\leq C\left(\intq e^{2s\varphi}|F|^2\   +(\lambda s)^{3}\intw e^{2s\varphi}\sigma^{3}v^{2}\   \right),\label{carl_jlr}
	\end{equation}
	where the constants $C, \lambda_0,s_0$  only depend on $\omega$, $a$, $\n{c}{L^\infty(\dom)}$ and $T$.
\end{thm}

\section{Preliminary Results}

\noindent

In this section we will  state some notations and results which are necessary to prove Theorem \ref{car-strong}. At first, we need to introduce some weighted spaces related to the function $a$,  namely

\begin{defn} [Weighted Sobolev spaces]
	Let us consider a real function $a=a(x)$ as in Hypotheses $A$.
	
	\begin{itemize}
		\item[(I)] For the (WDC), we set 
		\begin{align*}
		& \begin{multlined}[t][0.9\textwidth]
		H_a^1:= \{  u\in L^2(0,1);\ u\mbox{ is absolutely continuous in } [0,1],\\
		\sqrt{a}u_x\in L^2(0,1) \mbox{ and } u(1)=u(0)=0\},
		\end{multlined}\\	
		\end{align*}
		equipped with the natural norm
		\[ \|u\|_{H_a^1}:=\left( \|u\|_{L^2(0,1)}^2+\|\sqrt{a}u_x\|_{L^2(0,1)}^2 \right) ^{1/2} .\]
		
		\item[(II)] For the (SDC),
		\begin{align*}
		& \begin{multlined}[t][0.9\textwidth]
		H_a^1:= \{  u\in L^2(0,1);\ u\mbox{ is absolutely continuous in } (0,1],\\
		\sqrt{a}u_x\in L^2(0,1) \mbox{ and } u(1)=0\},
		\end{multlined}\\
		\end{align*}  
		and the norm keeps the same;
		
		\item[(III)] In both situations, the (WDC) and the (SDC), 
		\[
		H_a^2:= \{  u\in H_a^1;\ au_x\in H^1(0,1) \}
		\]
		with the norm
		$\|u\|_{H_a^2}:=\left( \|u\|_{H_a^1}^2+\|(au_x)_x\|_{L^2(0,1)}^2 \right) ^{1/2}$. 
	\end{itemize}
\end{defn}

Alabau-Boussouira at al. in \cite{alabau2006carleman} introduced and studied some of the main properties of these spaces.

Now, we will state  a Hardy-Poincaré type inequality, whose proof can be found in \cite{alabau2006carleman}. It represents a powerful estimate in order to hand the degeneracy of the function $a$, related to \eqref{pb-lin}.
\begin{prop}[Hardy-Poincaré Inequality]\label{prop-HP}
	Let $\tilde{a} :[0,1]\longrightarrow \R$ be a continuous  function such that $\tilde{a} (0)=0$ and $\tilde{a} >0$ in $(0,1]$. The following statements hold:
	\begin{itemize}
		\item[(a)] If there exists $\theta \in (0,1)$ such that  the function $x\mapsto a(x)/x^\theta$ is nonincreasing in $(0,1]$, then there exists a constant $C_H>0$ such that 
		\begin{equation}\label{HP_ineq}
		\into\frac{a(x)}{x^2}w^2(x)  \leq C_H\into a(x)|w'(x)|^2  ,
		\end{equation}
		for any real function $w$ that is locally absolutely continuous on $(0,1]$, continuous at $0$, and satisfies
		\begin{center}
			$w(0)=0$ and $\dps\into a(x)|w'(x)|^2\   <+\infty$.
		\end{center}
		\item[(b)] 	If there exists $\theta\in (1,2)$ such that the function $x\mapsto a(x)/x^\theta$ is nondecreasing in a neighborhood of $x=0$, then there exists a constant $C_H>0$ such that \eqref{HP_ineq} is valid for any function $w$
		that is locally absolutely continuous in $(0,1]$, and satisfies
		\[w(1)=0\ \text{ and } \int_0^1 a(x)|w'|^2<+\infty.\]
	\end{itemize}
\end{prop}

\begin{rem}\label{rem21} 
	Notice that Hypothesis \ref{hip1} implies some other useful conditions:
	
	\begin{enumerate}
		\item[(a)]  The relation \eqref{prop_a} means that the function $x\mapsto \frac{x^r}{a(x)}$ is nondecreasing on $(0,1]$, for all $r\geq K$, not only for the (WDC), but also for the $(SDC)$. In particular, $\dps x^2/a(x)\leq 1/a(1)$, for all $x\in (0,1]$.

		\item[(b)] Particularly for the (SDC), the assumption \eqref{teta} means that the function $x\mapsto \frac{a(x)}{x^\theta}$  is nondecresing.

	\end{enumerate}
\end{rem}

 The wellposednes of \eqref{pb-lin}, established in \cite{alabau2006carleman}, is the following: 

\begin{prop}\label{prop-WP-lin}
	For all $h\in L^2(\domw)$ and $u_0\in L^2(0,1)$, there exists a unique weak solution $u\in C^0([0,T];L^2(0,1))\cap L^2(0,T;H_a^1)$ of \eqref{pb-lin}.
	Moreover, if $u_0\in H_a^1$, then
	\[u\in\mathcal{U}:= H^1(0,T;L^2(0,1))\cap L^2(0,T;H_a^2)\cap C^0([0,T];H_a^1), \]
	and there exists a positive constant $C_T$ such that
	\begin{equation}\label{ineq1}
	\sup_{t\in[0,T]}\left(\n{u(t)}{H_a^1}^2\right)
	+\int_0^T\left(\n{u_t}{L^2(0,1)}^2+\n{(au_x)_x}{L^2(0,1)}^2\right)
	\leq C_T\left(\n{u_0}{H_a^1}^2 +\n{h}{L^2(\domw)}^2 \right)
	\end{equation}
\end{prop}

\section{Proof of Theorem \ref{car-strong}} 
\noindent

We start proving a Carleman inequality for the following problem
\begin{equation}\lnum \label{pbA1}
\left\{\begin{array}{ll}
v_t+\left(a\left(x\right)v_x \right)_x=h(t,x), & (t,x)\in \dom, \\
v(t,1)=0,&  t\in  (0,T),\\
\begin{cases*}
v(t,0)=0, & \text{ (Weak) }\\
\text{or}\\
(av_x)(t,0)=0,&  \text{ (Strong) }
\end{cases*},&  t\in (0,T). \\
\end{array}\right.   \end{equation}

\begin{prop} \label{propA1}
	There exist $C>0$ and $\lambda_0,s_0>0$ such that every solution  $v$ of (\ref{pbA1}) satisfies, for all $s\geq s_0$ and $\lambda\geq \lambda_0$, 
	\begin{equation}\label{car-A2}
	\intq e^{2s\varphi}\left((s\lambda)\sigma av_x^2+(s\lambda)^{5/3}\sigma^{5/3}v^2 \right) \leq C\left(\intq e^{2s\varphi}|h|^2\   +(\lambda s)^3\intw e^{2s\varphi}\sigma^3v^2\   \right)
	\end{equation} 
	
\end{prop}

We just need to prove this proposition  for the Strong case, since the inequality \eqref{car-A2} is a consequence of that proved in \cite{jrl2016}. In the Strong case, the proof is essentially the same, we just pay attention to the case in which $K=1$. Actually, we just have to present a new proof of  Lemma A.9 of \cite{jrl2016}, the other lemmas remain the same. However, for the sake of convenience, we will reproduce the entire proof here.

The proof of Proposition \ref{propA1} relies on the change of variables $w=e^{s\varphi}v$. Notice that
\begin{align*}
& v_t=e^{-s\varphi}(-s\varphi_tw+w_t),\\
& (av_x)_x=e^{-s\varphi}(s^2\varphi^2_xaw-s(a\varphi_x)_xw-2sa\varphi_xw_x+(aw_x)_x).
\end{align*} 
Then, from \eqref{pbA1}, we obtain 

\[\begin{cases}
L^+w+L^-w=e^{s\varphi}h,  & (t,x)\in \dom, \\
w(t,1)=0,&  t\in  (0,T),\\
\begin{cases*}
w(t,0)=0, & \text{ (Weak) }\\
\text{or}\\
(aw_x)(t,0)=0,&  \text{ (Strong) }
\end{cases*},&  t\in (0,T), \\
w(x,0)=w(x,T)=0, & x\in (0,1),
\end{cases}\]
where
\[L^+w:=-s\varphi_t w+s^2\varphi_x^2aw+(aw_x)_x,\]
\[L^-w:=w_t-s(a\varphi_x)_xw-2sa\varphi_xw_x.\]

In this way,
\[\|L^+w\|^2+\|L^-w\|^2+2(L^+w,L^-w)=\|e^{s\varphi}h\|^2,\]
where $\|\cdot\|$ and $(\cdot,\cdot)$ denote the norm and the inner product in $L^2(\dom)$, respectively.

From now on, we will prove  Lemmas \ref{A1}--\ref{A18}. The proof of Proposition \ref{propA1} will be a consequence of these lemmas.
\begin{lem}\label{lemA2}
	\begin{multline*}
	(L^+w,L^-w)  =\frac{s}{2}\intq\varphi_{tt}w^2-2s^2\intq \varphi_{tx}a\varphi_xw^2 +s^3\intq a\varphi_x(a\varphi_x^2)_xw^2\\
	+s\intq(a\varphi_x)_{xx}aww_x
	+ 2s\intq(a\varphi_x)_xaw_x^2\\ -s\intq a\varphi_xa_xw_x^2 -s\int_0^T(a^2\varphi_xw_x^2)\big\vert_{x=0}^{x=1}
	\end{multline*}
\end{lem}
\begin{proof}
	From the definition of $L^+w$ and $L^-w$ we have
	\begin{align*}
	(L^+w,L^-w) = & \intq (-s\varphi_tw+s^2\varphi_x^2aw+(aw_x)_x)w_t +s^2\intq\varphi_tw((a\varphi_x)_xw+2a\varphi_xw_x)\\
	& -s^3\intq \varphi^2_xaw((a\varphi_x)_xw+2a\varphi_xw_x) -s\intq(aw_x)_x((a\varphi_x)_xw+2a\varphi_xw_x)\\
	\phantom{(L^+w,L^-w)} = & I_1+I_2+I_3+I_4.
	\end{align*}
	
	Integrating by parts, we obtain
	\[I_1=\frac{s}{2}\intq(\varphi_{tt}-2sa\varphi_x\varphi_{xt})w^2,\]
	\[I_2=-s^2\intq \varphi_{tx}a\varphi_xw^2,\]
	\[I_3=s^3\intq a\varphi_x(a\varphi_x^2)_xw^2\]
	and
	\begin{equation*} I_4=s\intq(a\varphi_x)_{xx}aww_x+2s\intq(a\varphi_x)_xaw_x^2
	-s\intq(a\varphi_x)a_xw_x^2-s\int_0^T(a^2\varphi_xw_x^2)\big\vert_{x=0}^{x=1},
	\end{equation*}
	which imply the desired result.

\end{proof}
\begin{lem}\label{A1}
	$\dps-s\int_0^Ta^2\varphi_xw_x^2\big\vert_{x=0}^{x=1}\geq 0$
\end{lem}
\begin{proof} Since $\psi'(x)=x/a$, if $x\in [0,\alpha')$ and $\psi'(x)=-x/a$, if $x\in (\beta'1]$, we have
	\begin{align*} 
	-s\int_0^Ta^2\varphi_xw_x^2\big\vert_{x=0}^{x=1}=-s\lambda\int_0^Ta^2\psi'\sigma w_x^2\big\vert_{x=0}^{x=1}\geq 0.
	\end{align*}
	
\end{proof}
\begin{lem}\label{A2}
	
	\begin{equation*}
	s^3\intq a\varphi_x(a\varphi_x^2)_xw^2\geq C \lambda^4s^3\intq a^2|\psi'|^4\sigma^3w^2-Cs^3\lambda^3\intwl \sigma^3w^2
	+Cs^3\lambda^3\inta\frac{x^2}{a}\sigma^3w^2.
	\end{equation*}
\end{lem}
\begin{proof} Firstly, we observe that
	\begin{align*}
	s^3\intq a\varphi_x(a\varphi_x^2)_xw^2 &= s^3\lambda^3\intq a\psi'(a(\psi')^2)_x\sigma^3w^2 +2s^3\lambda^4\intq a^2(\psi')^4\sigma^3w^2\\
	& = I_1+I_2.
	\end{align*}

	We can see that
	\[a\psi'(a(\psi')^2)_x=\begin{cases}
	\phantom{-}\frac{x^2}{a^2}(2a-xa'), & x\in (0,\alpha')\\
	-\frac{x^2}{a^2}(2a-xa'), & x\in (\beta',1),
	\end{cases}\]
	and \eqref{prop_a} implies $2a-xa'\geq (2-K)a$. Hence,
	\begin{align*}
	I_1 &=\begin{multlined}[t]
	s^3\lambda^3\inta a\psi'(a(\psi')^2)_x\sigma^3w^2+s^3\lambda^3\intwl a\psi'(a(\psi')^2)_x\sigma^3w^2\\+s^3\lambda^3\intb a\psi'(a(\psi')^2)_x\sigma^3w^2
	\end{multlined}	\\
	& \geq (2-K) s^3\lambda^3\inta \frac{x^2}{a}\sigma^3w^2-Cs^3\lambda^3\intwl\sigma^3w^2
	-C(2-K)s^3\lambda^3\intq a^2|\psi'|^4\sigma^3w^2
	\end{align*}
	
	We just sum $I_1$ and $I_2$,  and take $\lambda_0$ large  enough to obtain the desired inequality.
	
\end{proof}

\begin{lem}\label{A3}
	\noindent 
	
	\begin{equation*}\dps 2s\intq(a\varphi_x)_xaw^2_x\geq -C\intwl \sigma w_x^2+Cs\lambda^2\intq a^2(\psi')^2\sigma w_x^2
	+2s\lambda\inta a\sigma w_x^2
	\end{equation*}
\end{lem}
\begin{proof} Observe that
	\begin{equation}\label{ast}
	2s\intq(a\varphi_x)_xaw^2_x=2s\intq\lambda(a\psi')_xa\sigma w_x^2
	+2s\lambda^2\intq a^2(\psi')^2\sigma w_x^2
	\end{equation}
	
	Proceeding as in lemma before, we split the first integral over the intervals $[0,\alpha'], \omega'$ and $[\beta',1]$. Since $a^2(\psi')^2\geq Ca$ in  $[\beta',1]$ we can add the integral over $[\beta',1]$ to the last integral of \eqref{ast}, which gives us the result.
	
\end{proof}
\begin{lem}\label{A4}
	\begin{equation*}\dps-2s^2\intq\varphi_{tx}a\varphi_xw^2\geq -Cs^2\lambda^2\left(\inta \frac{x^2}{a}\sigma^3w^2+\intwl\sigma^3w^2\right.
	+\left.\intq a^2|\psi'|^4\sigma^3w^2\right)
	\end{equation*}
\end{lem}
\begin{proof} First of all,
	\begin{equation*}
	\left|2s^2\intq \varphi_{tx}a\varphi_xw^2\right|\leq 2s^2\lambda^2\intq a|\psi'|^2|\theta \theta'|\eta^2 w^2
	\leq Cs^2\lambda^2\intq a|\psi'|^2\sigma^3 w^2
	\end{equation*}
	
	As before, we split the last integral over the   intervals $[0,\alpha'], \omega'$ and $[\beta',1]$. The result comes from the boundedness of $a|\psi'|^2$  in $\omega'$ and from relations  $\psi'= x/a$ in $[0,\alpha']$ and $a|\psi'|^2\leq Ca^2|\psi'|^4$ in $[b',1]$.
	
\end{proof}

\begin{lem}\label{A5}
	\[-s\intq a\varphi_x a_xw_x^2\geq -K\lambda s \inta a\sigma w_x^2-c\lambda s \intwl \sigma w_x^2\]
\end{lem}

\begin{proof} In fact, from the definition of $\psi$, we obtain
	\begin{align*} 
	-s\intq a\varphi_x a_xw_x^2& =-s\lambda \intq aa_x\psi'\sigma w_x^2\\
	& \geq -Ks\lambda \inta a\sigma w_x^2-C\lambda s\intwl \sigma w_x^2,
	\end{align*}
	where we proceeded as in the proof of Lemma \ref{A4}.
\end{proof}
\begin{lem}\label{A9} 
	\begin{multline*}
	s\intq (a\varphi_x)_{xx}aw_xw \geq   -Cs^2\lambda^4 \intq a^2|\psi'|^4 \sigma^3w^2-C\lambda^2 \intq a^2|\psi'|^2\sigma w_x^2\\ -Cs^2\lambda^3 \intwl \sigma^3w^2
	-C\lambda\intwl \sigma w_x^2\\
	-Cs^2\lambda^3\inta \frac{x^2}{a}\sigma^3w^2-C\lambda \inta a \sigma w_x^2
	\end{multline*}
\end{lem}

\begin{proof}
	\begin{align*}
	s\intq (a\varphi_x)_{xx}aw_xw &  =s\lambda\intwl (a\psi')_{xx}a\sigma w_xw +2s\lambda^2\intq (a\psi')_{x}\psi'a\sigma w_xw \\
	&  +s\lambda^2\intq a^2\psi'\psi''\sigma w_xw  +s\lambda^3\intq a^2(\psi')^3\sigma w_xw\\
	& =I_1+I_2+I_3+I_4.
	\end{align*}
	
	The inequality will be obtained by estimating each one of these fours integrals. For $I_1$, we have
	\begin{align*}
	|I_1|& =\left|s\lambda\intwl (a\psi')_{xx}a\sigma w_xw\right|\leq Cs\lambda\intwl \sigma^2 |w_xw|\\
	& =Cs\lambda\intwl \sigma^{3/2}|w|\sigma^{1/2} |w_x|\leq Cs\lambda\intwl \sigma^3 w^2+Cs\lambda\intwl \sigma w_x^2.
	\end{align*}
	
	For $I_2$, we use the facts $\sigma\leq C\sigma^2$ and $x\leq Ca^2|\psi'|^3$ in $[\beta',1]$ to obtain
	\begin{align*}
	|I_2| & \leq Cs\lambda^2\inta x\sigma^2 |w w_x|+ Cs\lambda^2\intwl \sigma^2 |w w_x|
	+Cs\lambda^2\intb a^2|\psi'|^3\sigma^2 |ww_x|\\
	&\leq  C\inta \left(\frac{x}{\sqrt{a}}\sigma^{3/2}\lambda^{3/2}s|w|\right)(\sqrt{a}\sigma^{1/2}\lambda^{1/2}|w_x|) + C\intwl \left(\sigma^{3/2}\lambda^{3/2}s|w|\right)(\sigma^{1/2}\lambda^{1/2}|w_x|)\\
	& + C\intb \left(s\sigma^{3/2}\lambda^{2} a|\psi'|^2|w|\right)(\sigma^{1/2}a|\psi'||w_x|)\\
	& \leq C\lambda^3s^2\inta \frac{x^2}{a}\sigma^3w^2 +C\lambda \inta a\sigma w_x^2 + C\lambda^3s^2\intwl \sigma^3w^2 +C\lambda\intwl \sigma w_x^2\\
	&+ Cs^2\lambda^4 \intq a^2 |\psi'|^4\sigma^3w^2 +C\intq a^2|\psi'|^2\sigma w_x^2
	\end{align*}
	
	For $I_3$, since $a'\geq 0$, for  $x\in [0,\alpha']\cup [\beta',1]$, we observe that
	\[|a^2\psi'\psi''|=\left|x\left(\frac{a-xa'}{a}\right)\right|\leq x \left|1-\frac{xa'}{a}\right|\leq x\left(1+\frac{xa'}{a}\right)\leq x(1+k).  \]
	Hence, using again that $\sigma\leq C\sigma^2$, we get
	\begin{align*}
	|I_3| & 
	\begin{multlined}[t]
	\leq s\lambda^2 \inta \left|x\left(\frac{a-xa'}{a}\right)\right|\sigma |ww_x|+Cs\lambda^2\intwl \sigma^2 |ww_x|\\
	+ s\lambda^2 \intb  \left|x\left(\frac{a-xa'}{a}\right)\right|\sigma |ww_x|
	\end{multlined}\\
	& \leq Cs\lambda^2\inta x\sigma^2 |ww_x|+Cs\lambda^2\intwl \sigma^2 |ww_x| + Cs\lambda^2\intb x\sigma |ww_x|.
	\end{align*} 
	So, we get the same estimate for $I_2$. Finally,
	\begin{equation*}
	|I_4| \leq \intq |s\lambda^2 a(\psi')^2\sigma^{3/2}w| |\lambda a \psi'\sigma^{1/2}w_x|\leq  Cs^2\lambda^4\intq a^2|\psi'|^4\sigma^3w^2+C\lambda^2\intq a^2|\psi'|\sigma w_x^2,
	\end{equation*}
	and the proof is complete.
\end{proof}
\begin{lem}\label{A10} 
	\begin{align*}
	\frac{s}{2}\intq \varphi_{tt}w^2\geq & -Cs\inta \sigma a w_x^2-Cs^2\lambda^2\inta \frac{x^2}{a}\sigma^3w^2 -Cs\intwl \sigma w_x^2\\ 
	&-C\lambda^2s^2\intwl \sigma^3w^2-Cs\intq a^2|\psi'|^2\sigma w_x^2-Cs^2\lambda^2\intq a^2|\psi'|^4\sigma^3w^2
	\end{align*}
\end{lem}
\begin{proof}
	Firstly, since $|\varphi_{tt}| \leq C\sigma^{3/2} $, we have that
	\begin{equation*}\label{eqA9}
	\left|\frac{s}{2}\intq \varphi_{tt}w^2\right|\leq Cs\intq \sigma^{3/2}w^2.
	\end{equation*}
	Therefore, we just need to bound this last integral. To do that, we will treat two separately cases,  $K\neq 1$ and $K=1$.

	For $k\neq 1$, we apply  Hardy-Poincar\'e inequality, to take
	\begin{align*}
	\intq \sigma^{3/2}w^2& \leq  \intq\left(\sigma^{1/2}\frac{\sqrt{a}}{x}w\right)\left(\sigma\frac{x}{\sqrt{a}}w\right) \leq \intq \sigma\frac{a}{x^2}w^2+\intq \sigma^2\frac{x^2}{a}w^2\\
	& \leq \intq \sigma aw^2_x+\intq\sigma^3\frac{x^2}{a}w^2
	\end{align*}
	
	Again, the two last intervals can be decomposed in  $[0,\alpha']$, $\omega'$ and $[\beta',1]$. At this point,  relations 
	$$\ a\leq Ca^2|\psi'|^2  \mbox{ and } \frac{x^2}{a}\leq C a^2|\psi'|^4, \mbox{ in } [\beta',1],$$
	give us the result.
	
	For $k=1$, Hardy-Poincaré inequality is not valid, since assumption \eqref{teta} does not give us $\theta \in (1,2)$ required in hypothesis in Proposition \ref{prop-HP}.  Therefore, we will define a function $p=p(x)$ which the Hardy-Poincaré inequality holds. 
	
	Indeed, define $p(x):=(a(x)x^4)^{1/3}$ and let $\theta \in (0,1)$ given by \eqref{teta}. If  we take  $q=\frac{4+\theta}{3}$, we can see that $q\in (1,2)$ and  the function $x\mapsto (p(x)/x^q)$ is  nondecreasing in a neighborhood of $x=0$, hence $p=p(x)$ satisfies the conditions of Proposition \ref{prop-HP}. 
	
	Let $\eta^\ast=\dps\max_{x\in[0,1]}\eta (x)$, since $\sigma(t,x) =\theta(t)\eta(x)$, $p(x)\leq C a(x)$ and  $\eta(x)\geq 1$ for all $x\in [0,1]$, we have that
	\begin{equation*}
	\into\left(\frac{a}{x^2}\right)^{1/3}\sigma w^2\leq {\eta^\ast}\theta(t)\into\left(\frac{a}{x^2}\right)^{1/3}w^2 \\
	=\eta^\ast\theta(t)\into \frac{p}{x^2}w^2\leq C\eta^\ast\theta(t) \into pw^2 \leq C\into \sigma a w_x^2.
	\end{equation*}
	
	From this inequality and using H\"older and Young inequalities for $p=4/3$ and $q=4$, we finally obtain that	
	\begin{align*}
	\intq \sigma^{3/2}w^2& =\intq \left(\frac{a^{1/4}}{x^{1/2}}\sigma^{3/4}w^{3/2}\right) \left(\frac{x^{1/2}}{a^{1/4}}\sigma^{3/4}w^{1/2}\right)\\
	& \leq  \left(\intq \frac{a^{1/3}}{x^{2/3}}\sigma w^2\right)^{3/4} \left(\intq \frac{x^{2}}{a}\sigma^{3}w^{2}\right)^{1/4}\\
	& \leq  \left(\intq \left(\frac{a}{x^2}\right)^{1/3}\sigma w^2\right)^{3/4} \left(\intq \frac{x^{2}}{a}\sigma^{3}w^{2}\right)^{1/4}\\
	& \leq  \left(\intq\sigma a w_x^2\right)^{3/4} \left(\intq \frac{x^{2}}{a}\sigma^{3}w^{2}\right)^{1/4}\\
	& \leq C\left( \intq\sigma a w_x^2 + \intq \frac{x^{2}}{a}\sigma^{3}w^{2}\right),
	\end{align*}
	where this last two integral are the same obtained in the case $K\neq 1$.
\end{proof}
\begin{lem}\label{A17}
	\begin{multline*}
	s^3\lambda^3 \inta \frac{x^2}{a}\sigma^3w^2+s\lambda \inta \sigma a w_x^2 +
	s^3\lambda^4\intq a^2|\psi'|^4\sigma^3w^2+s\lambda^2\intq a^2|\psi'|^2\sigma w_x^2\\
	\leq C\left(\intq e^{2s\varphi}|h|^2+s^3\lambda^3\intwl\sigma^3w^2+\lambda s \intwl \sigma w_x^2\right)
	\end{multline*}
\end{lem}
\begin{proof}
	From Lemmas \ref{lemA2}-\ref{A10}, we have
	\begin{multline*}
	(L^+w,L^-w)\geq  C\Bigg( s^3\lambda^3 \inta \frac{x^2}{a}\sigma^3w^2+s\lambda \inta \sigma a w_x^2 +
	\lambda^4s^3\intq a^2|\psi'|^4\sigma^3 w^2\\
	+s\lambda^2\intq a^2|\psi'|^2\sigma w_x^2 -s^3\lambda^3\intwl \sigma^3w^2-\lambda s \intwl \sigma w_x^2\Bigg).
	\end{multline*}
	
	Hence, 
	\begin{align*}
	& \begin{multlined}[t]
	C\Bigg( s^3\lambda^3 \inta \frac{x^2}{a}\sigma^3w^2+s\lambda \inta \sigma a w_x^2
	+ \lambda^4s^3\intq a^2|\psi'|^4\sigma^3 w^2\\ 
	+s\lambda^2\intq a^2|\psi'|^2\sigma w_x^2
	-s^3\lambda^3\intwl \sigma^3w^3-\lambda s \intwl \sigma w_x^2\Bigg)
	\end{multlined}\\
	& \leq  \nn{L^+w}^2+\nn{L^-w}^2+2(L^+w,L^-w) \leq \nn{e^{s\varphi}h}^2,
	\end{align*}
	following the result.
\end{proof}

Now, we intend to prove a suitable inequality which will imply Proposition \ref{propA1}. In order to do that, we recall that $v=e^{-s\varphi}w$.
\begin{lem}\label{A18}
	\begin{multline*}
	s^3\lambda^3 \inta e^{2s\varphi} \frac{x^2}{a} \sigma^3 v^2+s\lambda \inta e^{2s\varphi} \sigma av_x^2 
	\\	+s^3\lambda^4 \intq e^{2s\varphi} a^2|\psi'|^4\sigma^3v^2+s\lambda^2\intq e^{2s\varphi}a^2|\psi'|^2\sigma v_x^2\\
	\leq C\left(\intq e^{2s\varphi} |h|^2 +\lambda^3s^3 \intw e^{2s\varphi}\sigma^3v^2\right)
	\end{multline*}
\end{lem}

\begin{proof}
	Since $v=e^{-s\varphi}w$, we have
	\begin{align*}
	&  e^{s\varphi}v_x=-s\lambda \psi' \sigma w+w_x 
	\end{align*}
	which implies
	\begin{align*}
	e^{2s\varphi}s\lambda^2|\psi'|^2a^2\sigma v_x^2
	& =(s\lambda^2|\psi'|^2a^2\sigma)e^{2s\varphi}v_x^2 \leq C(s\lambda^2|\psi'|^2a^2\sigma)(s^2\lambda^2|\psi'|^2\sigma^2w^2+w_x^2)\\
	& \leq C(s^3\lambda^4|\psi'|^4\sigma^3a^2w^2+s\lambda^2|\psi'|^2a^2\sigma w_x^2)
	\end{align*}
	Besides that,
	\begin{align*}
	& w_x=s\varphi_xe^{s\varphi}v+e^{s\varphi}v_x\Rightarrow w_x^2\leq C(s^2\lambda^2|\psi'|^2\sigma^2e^{2s\varphi}v^2+e^{2s\varphi}v_x^2)\nonumber\\
	\Rightarrow & w_x^2\leq C(s^2\lambda^2\sigma^2e^{2s\varphi}v^2+e^{2s\varphi}av_x^2), \mbox { in } \omega'
	\end{align*}
	
	Hence, from Lemma \ref{A17}, we get
	\begin{align}\label{A12}
	& \begin{multlined}[t]
	s^3\lambda^3 \inta e^{2s\varphi} \frac{x^2}{a} \sigma^3  v^2+s\lambda \inta e^{2s\varphi} \sigma av_x^2 +	s^3\lambda^4 \intq e^{2s\varphi} a^2|\psi'|^4\sigma^3v^2\nonumber\\
	+s\lambda^2\intq e^{2s\varphi}a^2|\psi'|^2\sigma v_x^2\nonumber
	\end{multlined}\\
	&\begin{multlined}[t]
	\leq C\bigg(  s^3\lambda^3 \inta \frac{x^2}{a} \sigma^3  w^2+s\lambda \inta  \sigma aw_x^2  +	s^3\lambda^4 \intq a^2|\psi'|^4\sigma^3w^2\\
	+s\lambda^2\intq a^2|\psi'|^2\sigma w_x^2\bigg)
	\end{multlined}\nonumber\\
	&\leq C\left(\intq e^{2s\varphi}|h|^2+s^3\lambda^3\intwl\sigma^3w^2+\lambda s \intwl \sigma w_x^2\right)\nonumber\\
	& \leq C\left(\intq e^{2s\varphi}|h|^2+s^3\lambda^3\intwl e^{2s\varphi}\sigma^3v^2+\lambda s \intwl e^{2s\varphi}  \sigma av_x^2\right)
	\end{align}
	
	To complete the proof we will estimate the last integral of \eqref{A12}. Firstly, let us take $\chi\in C_0^\infty(\omega)$  such that $0\leq \chi \leq 1$ and $\chi\equiv 1$ in $\omega'$. Multiplying equation in \eqref{pbA1} by $\lambda s e^{2s\varphi}\sigma v \chi$ and integrating over $\dom$, we obtain
	\begin{equation}\label{A13}
	\lambda s\intq e^{2s\varphi}\sigma v v_t\chi+\lambda s\intq e^{2s\varphi} \sigma (av_x)_xv\chi 
	=\lambda s \intq e^{2s\varphi}\sigma hv\chi.
	\end{equation}
	
	We can see that
	\begin{multline}\label{A14}
	\left|\intq e^{2s\varphi}\sigma v v_t\chi\right|=\left|\frac{1}{2}\intq e^{2s\varphi}\sigma \frac{d}{dt}v^2\chi\right|=\left|-\frac{1}{2}\intq (e^{2s\varphi}\sigma \chi)_{_t}v^2\right| \\
	=\left|-\frac{1}{2}\intq \chi e^{2s\varphi}(2s\varphi_t\sigma+\sigma_t)v^2\right|\leq Cs\intw e^{2s\varphi} \sigma^3v^2.
	\end{multline}
	
	And, analogously,
	\begin{align*}
	\intq e^{2s\varphi}\sigma (av_x)_x v \chi=-\intq e^{2s\varphi}\sigma a v_x^2 \chi -\intq (e^{2s\varphi}\sigma \chi)_xav_xv.
	\end{align*}
	Since $\varphi_x\leq C\sigma$ and $\sigma_x\leq C\sigma$ in $\domw$, we get
	\begin{equation}\label{A15}
	\left|\intq (e^{2s\varphi}\sigma \chi)_xav_xv\right|\leq C\intw  e^{2s\varphi} \sigma^2 |av_x||v|.
	\end{equation}
	
	Now, from \eqref{A13}-\eqref{A15} we obtain
	\begin{align*}
	& \lambda s\intw e^{2s\varphi} \sigma a v_x^2\leq \lambda s\intq e^{2s\varphi} \sigma a v_x^2\chi\\
	& \leq \left|-\lambda s \intq \wei\sigma (a v_x)_xv\chi -\lambda s \intq (\wei \sigma \chi)_xav_xv\right|\\
	& \leq \lambda s \intq \wei \sigma |vv_t|\chi +\lambda s \intq \wei \sigma |hv|\chi +\lambda s \intq |(\wei \sigma\chi)_x||av_xv| \\
	& \leq C\lambda s^2\intw e^{2s\varphi} \sigma^3v^2 + \lambda s \intw (e^{s\varphi} h)(e^{s\varphi}\sigma v)+ C\lambda s\intw  e^{2s\varphi} \sigma^2 |av_x||v|\\
	& \leq C\lambda^3 s^3\intw e^{2s\varphi} \sigma^3v^2 + \frac{1}{2}\lambda s \intw e^{2s\varphi} h^2+\frac{1}{2}\lambda s \intw e^{2s\varphi} \sigma^2 v^2\\
	& \phantom{\lambda s \intq \wei \sigma |vv_t|\chi }+ C\lambda s\intw  (e^{s\varphi} \sigma^{1/2}a^{1/2}|v_x| )(e^{s\varphi} a^{1/2}\sigma^{3/2}|v|)\\
	& \leq C\lambda^3 s^3\intw e^{2s\varphi} \sigma^3v^2 + C \intw e^{2s\varphi} h^2\\
	& \phantom{\lambda s \intq \wei \sigma |vv_t|\chi }+ \ep C\lambda s\intw  e^{2s\varphi} \sigma a v_x^2+C_\ep \intw e^{2s\varphi} a\sigma^3 v^2\\
	& \leq C\lambda^3 s^3\intw e^{2s\varphi} \sigma^3v^2 + C \intw e^{2s\varphi} h^2+ \ep C\lambda s\intw  e^{2s\varphi} \sigma a v_x^2.
	\end{align*}
	Hence, taking $\ep=1/2C$, we get
	\begin{equation*}
	\lambda s\intw e^{2s\varphi} \sigma a v_x^2\leq  C\left(  \intw e^{2s\varphi} h^2+\lambda^3 s^3\intw e^{2s\varphi} \sigma^3v^2 \right).
	\end{equation*}
	It last inequality combined with \eqref{A12} completes the proof.	
\end{proof}
Now we are ready to prove Proposition \ref{propA1}.

\begin{proof}[Proof of Proposition \ref{propA1}]
	\noindent 
	
	
	Let $p,q>1$ and $\beta := \frac{1}{p}+\frac{3}{q}$  such that $\frac{1}{p}+\frac{1}{q}=1$. These numbers will be precise later depending on the case $K\neq 1$ or $K=1$.
	
	Using H\"older and Young inequalities,  we have that
	
	\begin{align*}
	(\lambda s)^{\beta}\intq \wei \sigma^{\beta}v^2& =(\lambda s)^{\beta}\intq \sigma^{\beta}w^2\\
	& = \intq \left((\lambda s)^3 \sigma^{3}\frac{x^2}{a} w^2\right)^{1/q} \left((\lambda s)  \sigma \left(\frac{a}{x^2}\right)^{p/q} w^2\right)^{1/p}\\
	& \leq  \left(\intq (\lambda s)^3 \sigma^{3}\frac{x^2}{a} w^2\right)^{1/q}  \left(\intq(\lambda s)  \sigma \left(\frac{a}{x^2}\right)^{p/q} w^2\right)^{1/p}\\
	& \leq C\left( s^3\lambda^3\intq \sigma^3\frac{x^2}{a}w^2+s\lambda \intq \sigma\left(\frac{a}{x^2}\right)^{p/q}w^2\right)\\
	& =C(I_1+I_p).
	\end{align*}

	Now, let us estimate $I_1$ and $I_p$ taking into account the terms of the inequality given by Lemma \ref{A17}.
	
	Splitting $I_1$ over the   intervals $[0,\alpha'], \omega'$ and $[\beta',1]$, and taking into account that $x^2/a$ is bounded in $\omega'$ and $x^2/a\leq a^2|\psi'|^2$ in $[b',1]$, we use Lemma \ref{A17} to obtain that
	\begin{align*}
	I_1 & \leq s^3\lambda^3\inta \sigma^3\frac{x^2}{a}w^2+Cs^3\lambda^3\intwl \sigma^3 w^2+Cs^3\lambda^4\intq \sigma^3 a^2|\psi'|^4w^2\\
	& \leq C\left(\intq \wei h^2+s^3\lambda^3\intwl \sigma^3 w^2+\lambda s\intwl \sigma w_x^2\right)
	\end{align*}
	
	In order to estimate $I_p$, we will consider two cases $K\neq 1$ and $K=1$. 
	
	If $K\neq 1$, we choose $p=2$ and  we can apply Hardy-Poincar\'e inequality as following
	\begin{align*}
	& I_p= s\lambda\intq \frac{a}{x^2}(\sigma^{1/2}w)^2\leq Cs\lambda\intq a(\sigma^{1/2}w)_x^2\\
	& =Cs\lambda \intq a\left(\frac{1}{2}\sigma^{-1/2}\sigma_xw+ \sigma^{1/2}w_x\right)^2\\
	& \leq Cs\lambda \intq a\sigma^{-1}\sigma_x^2w^2+Cs\lambda \intq a\sigma w_x^2\\
	& \leq Cs\lambda^3 \intq a|\psi'|^2\sigma w^2+Cs\lambda \intq a\sigma w_x^2\\
	& \leq Cs^3\lambda^3 \inta \frac{x^2}{a}\sigma^3 w^2+Cs^3\lambda^3 \intwl \sigma^3 w^2+Cs^3\lambda^4 \intb a^2|\psi'|^4\sigma^3 w^2\\
	&\phantom{\leq}+Cs\lambda \inta a\sigma w_x^2+Cs\lambda \intwl \sigma w_x^2+Cs\lambda^2 \intb a^2|\psi'|^2\sigma w_x^2\\
	& \leq C\left(\intq \wei h^2+s^3\lambda^3\intwl \sigma^3 w^2+\lambda s\intwl \sigma w_x^2\right)
	\end{align*}
	
	If $K=1$, we will proceed as in Lemma \ref{A10}, where have had to define a suitable function in order to apply Hardy-Poincaré inequality.
	
	In this case, let us choose $p=3/2$ and define $b(x):=\sqrt{a(x)}x$. Let $\theta \in (0,1)$ given by \eqref{teta}. If  we take  $q=\frac{\theta}{2}+1$, we can see that $q\in (1,3/2)$ and  the function $x\mapsto (b(x)/x^q)$ is  nondecreasing in a neighborhood of $x=0$, hence $b$ satisfies the conditions of Proposition \ref{prop-HP}. 
	
	Recalling that $\eta^\ast=\dps\max_{x\in[0,1]}\eta (x)$, since $\sigma(t,x) =\theta(t)\eta(x)$, $b(x)\leq C a(x)$ and  $\eta(x)\geq 1$ for all $x\in [0,1]$, we have that
	\begin{multline*}
	\into\left(\frac{a}{x^2}\right)^{p/q}\sigma w^2\leq\into\left(\frac{a}{x^2}\right)^{1/2}\sigma w^2\leq {\eta^\ast}\theta(t)\into\left(\frac{a}{x^2}\right)^{1/3}w^2 \\
	=\eta^\ast\theta(t)\into \frac{b}{x^2}w^2\leq C\eta^\ast\theta(t) \into bw^2 \leq C\into \sigma a w_x^2.
	\end{multline*}
	Hence,
	\[I_p=s\lambda \intq \sigma\left(\frac{a}{x^2}\right)^{1/2}w^2 \leq C\intq \sigma a w_x^2,\]
	which is one of the integral obtained in the case $p=2$.

	Thus, note that for $p=2$, $\beta=2$ and for $p=3/2$, $\beta=5/3$. Therefore, in both cases, we have that
	\begin{multline*}
	(\lambda s)^{5/3}\intq \wei \sigma^{5/3}v^{5/3}\leq (\lambda s)^{\beta}\intq \wei \sigma^{\beta}v^{\beta}\\
	\leq I_1+I_p \leq C\left(\intq \wei h^2+s^3\lambda^3\intwl \sigma^3 w^2+\lambda s\intwl \sigma w_x^2\right).
	\end{multline*}
	Proceeding exactly as in the proof of Lemma \ref{A18}, we achieve
	\begin{equation*}\label{A20}
	(\lambda s)^{5/3}\intq \wei \sigma^{5/3}v^{5/3}\leq  C\left(  \intw e^{2s\varphi} h^2+\lambda^3 s^3\intw e^{2s\varphi} \sigma^3v^2 \right),
	\end{equation*}
	and the result given by Lemma \ref{A18} gives us
	\begin{align*}
	s\lambda \intq \wei a\sigma v_x^2 &\leq s\lambda \inta \wei a\sigma v_x^2+s\lambda \intwl \wei a\sigma v_x^2	+s\lambda \intb \wei a^2|\psi'|^2\sigma v_x^2\\
	&\leq  C\left(  \intw e^{2s\varphi} h^2+\lambda^3 s^3\intw e^{2s\varphi} \sigma^3v^2 \right).
	\end{align*}
	
	Therefore, this last two estimates conclude the proof of Proposition \ref{propA1}.
	
\end{proof}

\begin{proof}[Proof of Theorem \ref{car-strong}] 
	If $v$ is a solution of \eqref{adj-jlr}, then $v$ is also a solution of \eqref{pbA1} with $h=F+cv$. In this case, applying Propostion \ref{propA1}, there exist $C>0$, $\lambda_0>0$ and $s_0>0$ such that $v$ satisfies, for all $s\geq s_0$ and $\lambda\geq \lambda_0$, 
	\begin{equation} \label{A19}
	\intq e^{2s\varphi}\left((s\lambda)\sigma av_x^2+(s\lambda)^{5/3}\sigma^{5/3}v^2 \right) 	\leq C\left(\intq e^{2s\varphi}|h|^2\   +(\lambda s)^3\intw e^{2s\varphi}\sigma^3v^2\   \right).
	\end{equation} 
	
	Recalling that $c\in L^{\infty}(\dom)$ and $\sigma\geq C>0$, we can see that
	\begin{align*}
	\intq e^{2s\varphi}|h|^2& =\intq e^{2s\varphi}|F+cv|^2\\
	& \leq C\intq e^{2s\varphi}|F|^2+C\n{c}{\infty}^2\intq e^{2s\varphi}|v|^2\\
	& \leq C\intq e^{2s\varphi}|F|^2+C\intq e^{2s\varphi}\sigma^{5/3}|v|^2.
	\end{align*}
	
	Therefore, taking $\lambda_0$ and $s_0$ large enough, the last integral can be absorbed by the left-hand side of \eqref{A19}, which complete the proof.
\end{proof}


\bibliography{references}

\end{document}